\documentclass[12pt,a4paper]{amsart}
\usepackage[top=1.13in, bottom=1.13in, left=1.15in, right=1.15in]{geometry}
\usepackage[T1]{fontenc}
\usepackage{latexsym}
\usepackage{amsmath}
\usepackage{shuffle}
\usepackage{amssymb, amscd, amsthm, mathrsfs}
\usepackage[all]{xy}
\usepackage[dvips]{graphicx}

\numberwithin{equation}{section}

\newtheorem{definition}{Definition}
\newtheorem{proposition}{Proposition}
\newtheorem{theorem}{Theorem}
\newtheorem{lemma}{Lemma}
\newtheorem{corollary}{Corollary}
\newtheorem{remark}{Remark}
\newtheorem{example}{Example}

\newcommand{\ra}{\rightarrow}

\newcommand{\vp}{\varphi}

\newcommand{\ts}{\otimes}
\newcommand{\s}{\sigma}

\newcommand{\mc}{\mathbb{K}}

\newcommand{\wt}{\widetilde}

\newcommand{\one}{\mathbf{1}}

\newcommand{\GVB}{\operatorname*{GVB}}
\newcommand{\VB}{\operatorname*{VB}}

\newcommand{\id}{\operatorname*{id}}
\newcommand{\sspan}{\operatorname*{span}}
\newcommand{\im}{\operatorname*{im}}
\newcommand{\des}{\operatorname*{des}}
\newcommand{\mm}{\operatorname*{\bold{M}}}

\newcommand{\End}{\operatorname*{End}}
\newcommand{\Hom}{\operatorname*{Hom}}
\newcommand{\QS}{\operatorname*{QS}}

\numberwithin{equation}{section}

\begin{document}

\title[Generalized Matsumoto-Tits sections]{Generalized Matsumoto-Tits sections and quantum quasi-shuffle algebras}

\author{Xin FANG}
\address{Mathematisches Institut, Universit\"{a}t zu K\"{o}ln, Weyertal 86-90, D-50931, K\"{o}ln, Germany.}
\email{xfang@math.uni-koeln.de}
\maketitle 

\begin{abstract}
In this paper generalized Matsumoto-Tits sections lifting permutations to the algebra associated to a generalized virtual braid monoid are defined. They are then applied to study the defining relations of the quantum quasi-shuffle algebras via the total symmetrization operator.
\end{abstract}

\section{Introduction}

\subsection{Motivations}

Quantum quasi-shuffle products (\cite{JRZ11}), as generalizations of the quasi-shuffle products (\cite{Hoff}, \cite{NR79}, \cite{Car72}), quantum shuffle products (\cite{Ros98}, \cite{Green}) and shuffle products (\cite{EM53}), unify several important but apparently different algebraic structures such as:
\begin{itemize}
\item Product of iterate integrals (\cite{Ree}, \cite{Chen});
\item Free Rota-Baxter algebras (\cite{Car72}, \cite{GK00});
\item Positive parts of quantum groups (\cite{Ros95}, \cite{Ros98}, \cite{Green}); the entire quantum group (\cite{FR12});
\item Multiple zeta values (\cite{Hoff}) and their $q$-deformed versions (\cite{HI});
\item Quasi-symmetric functions and their $q$-deformed versions (\cite{TU96});
\item Free dendriform and tridendriform algebras (\cite{Lod});
\item Hopf algebra of rooted trees (\cite{CK98});
\item (Conjecturally) Fomin-Kirillov algebras (\cite{FK99}, \cite{MS99}).
\end{itemize}
These products are originally defined using either universal properties \cite{JR12} or by inductive formulae \cite{JRZ11}. Its combinatorial nature is first studied by Hoffman \cite{Hoff}, Guo-Keigher \cite{GK00} and Ebrahimi-Fard-Guo \cite{EG06} in the classical case (quasi-shuffle products) and henceforth generalized by Jian \cite{J13} to the quantum case using the notion of mixable shuffles \cite{EG06}.
\par
Another different combinatorial description of the quantum quasi-shuffle products is discovered in \cite{F13} via lifting shuffles to the generalized virtual braid monoids, revealing symmetries behind these combinatorial structures. More precisely, for positive integers $p+q=n$, let $\mathfrak{S}_{p,q}\subset \mathfrak{S}_n$ be the set of $(p,q)$-shuffles and $\GVB_n^+$ be the monoid associated to the generalized virtual braid group $\GVB_n$ (see Section \ref{Sec:2} for the definition). A set-theoretical map $Q_{p,q}$ from $\mathfrak{S}_{p,q}$ to $\mc[\GVB_n^+]$ is constructed such that the quantum quasi-shuffle product of two pure tensors of degrees $p$ and $q$ respectively is encoded by the action of $\sum_{s\in\mathfrak{S}_{p,q}}Q_{p,q}(s)$. The braid group $\mathfrak{B}_n$ is a quotient of $\GVB_n$, the composition of this projection and $Q_{p,q}$ gives a map $\mathfrak{S}_{p,q}\ra\mc[\mathfrak{B}_n^+]$, which coincides with the classical Matsumoto-Tits section. It is for this reason that the map $Q_{p,q}$ is called a generalized Matsumoto-Tits section.
\par
It is natural to ask for non-trivial extensions of $Q_{p,q}$ to the whole $\mathfrak{S}_{n}$, generalizing the classical Matsumoto-Tits section. The goal of this paper is to solve this question and to study its applications to the total symmetrization operator and to the defining ideal of a quantum quasi-shuffle algebra.

\subsection{Main ideas}

The idea of the construction of the generalized Matsumoto-Tits section for the entire $\mathfrak{S}_n$ is noticing that $\mathfrak{S}_{p,q}\backslash\{e\}$ coincides with elements in $\mathfrak{S}_n$ having only one descent at the position $p$. Once a permutation in $\mathfrak{S}_n$ is given, it can be decomposed according to its descents into shuffles and we can iterate the lifting of shuffles.
\par
In the situation of two descents, there are two different ways to decompose a permutation into shuffles, the associativity of the quantum quasi-shuffle product encodes the difference of these two liftings. By contrast with the classical Matsumoto-Tits section, there is no canonical way to lift permutations to generalized virtual braid groups.
\par
This construction allows us to define the total symmetrization operator which can be used to express the quantum quasi-shuffle algebra as its image, giving the definition of the defining ideal as its kernel. Finally we discuss how to lift the defining relations in the corresponding quantum shuffle algebra to the quantum quasi-shuffle algebra, where an algorithm is proposed.

\subsection{Outline of this paper}
In Section 2 motivations arising from the (quantum) shuffle algebras and the constructions in \cite{F13} about lifting shuffles are recalled. We tackle with the problem of lifting permutations of two descents in Section 3 and then apply it to study the combinatorial associativity of the quantum quasi-shuffle product. The problem of lifting a permutation of an arbitrary descent is solved in Section 4. As applications, we study the total symmetrization map, which allows us to define the defining ideal associated to a quantum quasi-shuffle algebra. The last section is devoted to the study of the degeneration of a quantum quasi-shuffle algebra to a quantum shuffle algebra, which is then applied to provide an algorithm on the lifting of the defining relations.

\section{Backgrounds, recollections and basic settings}\label{Sec:2}

\subsection{Definitions and notations}

For an integer $n\geq 2$, let $\mathfrak{S}_n$ denote the symmetric group acting on $n$ positions numerated by $1,\cdots,n$ via permuting their positions. For two integers $p,q>0$ satisfying $p+q=n$, a permutation $\s\in\mathfrak{S}_n$ is called a $(p,q)$-shuffle if $\s^{-1}(1)<\cdots<\s^{-1}(p)$ and $\s^{-1}(p+1)<\cdots<\s^{-1}(p+q)$, the set of $(p,q)$-shuffles in $\mathfrak{S}_n$ will be denoted by $\mathfrak{S}_{p,q}$.

\par
Let $\GVB_n$, $\VB_n$ and $\mathfrak{B}_n$ denote the generalized virtual braid group, the virtual braid group and the braid group respectively; $\mc[\GVB_n]$, $\mc[\VB_n]$ and $\mc[\mathfrak{B}_n]$ denote the corresponding group algebras. The group $\GVB_n$ (\cite{F13}) is generated by $\{\s_i,\xi_i|\ i=1,\cdots,n-1\}$ subject to the following relations
\begin{enumerate}
\item[(1).] for $|i-j|>1$, $
\s_i\s_j=\s_j\s_i$, $\s_i\xi_j=\xi_j\s_i$, $\xi_i\xi_j=\xi_j\xi_i$;
\item[(2).] for $1\leq i\leq n-2$, $
\s_i\s_{i+1}\s_i=\s_{i+1}\s_i\s_{i+1}$, $\xi_i\xi_{i+1}\xi_i=\xi_{i+1}\xi_i\xi_{i+1}$,
$\xi_i\s_{i+1}\s_i=\s_{i+1}\s_i\xi_{i+1}$, $\xi_{i+1}\s_i\s_{i+1}=\s_{i}\s_{i+1}\xi_i$.
\end{enumerate}
The group $\VB_n$ is the quotient of $\GVB_n$ by the normal sub-group generated by $\s_i^2$ for $i=1,\cdots,n-1$. Let $\{s_i|\ i=1,\cdots,n-1\}$ and $\{\s_i|\ i=1,\cdots,n-1\}$ be the generating sets of $\mathfrak{S}_n$ and $\mathfrak{B}_n$, respectively. The monoids associated to these groups are denoted by $\mathfrak{B}_n^+$, $\VB_n^+$ and $\GVB_n^+$; we let $\mc[\mathfrak{B}_n^+]$, $\mc[\VB_n^+]$ and $\mc[\GVB_n^+]$ denote the corresponding algebras.

\begin{remark}
Our notations on the virtual braid groups are different from those in Kauffman-Lambropoulou \cite{Kauf}: the roles of $\s_i$ and $\xi_i$ are swapped.
\end{remark}

Throughout this paper the following arrow notations will be applied.
\par
Fix two integers $1\leq m<n$ and $0\leq k\leq n-m$. Let $i_k$ denote the $\mc$-algebra map $\mc[\GVB_m]\ra\mc[\GVB_n]$ uniquely determined by: for $1\leq i\leq m-1$, 
$$i_k(\s_i)=\s_{i+k},\ \ i_k(\xi_i)=\xi_{i+k}.$$
The arrow notation $\uparrow$ will be used to denote the image under $i_k$: for any $\s\in\GVB_m$,
$$\s^{\uparrow k}:=i_k(\s).$$
Since $i_k$ is $\mc$-linear, the above notation will be applied when $\s$ is either in $\mc[\GVB_m]$ or in $\mc[\GVB_m^+]$.
\par
The algebras $\mc[\VB_n]$, $\mc[\mathfrak{B}_n]$ and $\mc[\mathfrak{S}_n]$ are quotients of $\mc[\GVB_n]$; it is clear that $i_k$ passes to these quotients and the arrow notation can be applied to these situations: for any $\s$ in any of these algebras, we let $\s^{\uparrow k}$ denote the image of $\s$ under $i_k$.
\par
This notation will also be used for local representations. It will be explained in the case of $\GVB_m$ and is applicable to all situations mentioned above. Let $\rho:\mc[\GVB_m]\ra\End(V^{\ts m})$ be the local representation of $\mc[\GVB_m]$ defined by 
$$\s_i\mapsto {\id}^{\ts (i-1)}\ts \s\ts{\id}^{(m-i-1)},\ \ \xi_i\mapsto {\id}^{\ts (i-1)}\ts \tau\ts{\id}^{(m-i-1)}\ \ 
\text{for some}\ \ \s,\tau\in\End(V^{\ts 2}).$$ 
For $1\leq k\leq n-m$, we let $\rho^{\uparrow k}$ denote the representation $\mc[\GVB_m]\ra\End(V^{\ts n})$ defined by:
$$\s_i\mapsto {\id}^{\ts (k+i-1)}\ts \s\ts{\id}^{(n-k-i-1)},\ \ \xi_i\mapsto {\id}^{\ts (k+i-1)}\ts \tau\ts{\id}^{(n-k-i-1)}$$
for some $\s,\tau\in\End(V^{\ts 2})$ as above.
\par
For a subset $S$ of $\GVB_m$ (resp. $\VB_m$, $\mathfrak{B}_m$, $\mathfrak{S}_m$), we will denote 
$$S^{\uparrow k}:=i_k(S)=\{\s^{\uparrow k}|\ \s\in S\}\subset{\GVB}_n\ \  (\text{resp. }{\VB}_n,\ \mathfrak{B}_n,\ \mathfrak{S}_n).$$
When $k=0$, we sometimes omit the arrow notation. For instance, for $p+q<n$, the subset $\mathfrak{S}_{p,q}^{\uparrow 0}$ will be denoted by $\mathfrak{S}_{p,q}$.

\subsection{Shuffle products}
To motivate the main construction of this paper we start with recollections on the shuffle product and its quantization from a combinatorial point of view.

Let $V$ be a vector space and $T(V)$ be the tensor vector space associated to $V$. We denote:
$$T^n(V):=V^{\ts n},\ \ T^{\leq n}(V):=\bigoplus_{k=0}^n T^{k}(V).$$ 
The shuffle product on $T(V)$ is defined as follows: for $v_1,\cdots,v_{p+q}\in V$,
$$(v_1\ts\cdots\ts v_p)\cdot (v_{p+1}\ts\cdots\ts v_{p+q})=\sum_{s\in\mathfrak{S}_{p,q}}s(v_1\ts\cdots\ts v_{p+q}).$$
The associativity of the shuffle product can be proved either in an algebraic way by applying the universal property or combinatorially by observing the following facts:
\begin{enumerate}
\item For $1\leq i<j< n$ and $v_1,\cdots,v_n\in V$,
$$((v_1\ts\cdots\ts v_i)\cdot (v_{i+1}\ts\cdots\ts v_j))\cdot (v_{j+1}\ts\cdots\ts v_n)=\sum_{t\in\mathfrak{S}_{j,n-j}}\sum_{s\in\mathfrak{S}_{i,j-i}}ts(v_1\ts\cdots\ts v_n),$$
$$(v_1\ts\cdots\ts v_i)\cdot ((v_{i+1}\ts\cdots\ts v_j)\cdot (v_{j+1}\ts\cdots\ts v_n))=\sum_{t\in\mathfrak{S}_{i,n-i}}\sum_{s\in\mathfrak{S}_{j-i,n-j}^{\uparrow i}}ts(v_1\ts\cdots\ts v_n).$$
\item Suppose that $p+q\leq n$; we denote the partial symmetrization operator $S_{p,q}=\sum_{s\in\mathfrak{S}_{p,q}}s$ in the group algebra $\mc[\mathfrak{S}_n]$. It is clear that for $1\leq i<j<n$,
\begin{equation}\label{eq1}
S_{j,n-j}S_{i,j-i}=S_{i,n-i}S_{j-i,n-j}^{\uparrow i}.
\end{equation}
This formula can be proved by examing the descents (see Lemma \ref{Lem:bij2} and Lemma \ref{Lem:decomp}).
\end{enumerate}

This combinatorial point of view allows us to look at the shuffle product purely inside of the algebra of symmetric groups via descent algebras, i.e., the shuffle product is the "representation" of this structure.
\par
The tensor space $T(V)$, endowed with the shuffle product, is called the shuffle algebra associated to $V$; it will be denoted by $\text{Sh}(V)$. The sub-algebra $S(V)$ of $\text{Sh}(V)$ generated by $V$ is isomorphic to the symmetric algebra on $V$.

\subsection{Quantum shuffle algebras}\label{Sec:2.3}
The quantum shuffle algebra (\cite{Ros95}, \cite{Ros98}) is a quantization of the shuffle algebra by considering a supplementary braid structure on the vector space.
\par
Let $(V,\s)$ be a braided vector space. The quantum shuffle product on the tensor vector space $T(V)$ is defined in the following way:
\begin{enumerate}
\item Let $T:\mathfrak{S}_n\ra\mathfrak{B}_n^+$ be the Matsumoto-Tits section: for $s\in\mathfrak{S}_n$ with a reduced expression $s=s_{i_1}\cdots s_{i_l}$, $T$ is the set-theoretic map sending $s$ to $T_s:=\s_{i_1}\cdots \s_{i_l}$, which is independent of the choice of the reduced expression. Its linear extension $\mc[\mathfrak{S}_n]\ra\mc[\mathfrak{B}_n^+]$ will also be denoted by $T$.
\item For $p,q\geq 1$ and $v_1,\cdots,v_{p+q}\in V$, the quantum shuffle product is defined by:
$$(v_1\ts\cdots\ts v_p)\bullet (v_{p+1}\ts\cdots\ts v_{p+q})=\sum_{s\in\mathfrak{S}_{p,q}}T_s(v_1\ts\cdots\ts v_{p+q}).$$
\end{enumerate}

The original construction of this product and its associativity (\textit{loc.cit.}) rely on the universal property of the cotensor coalgebra associated to a Hopf bimodule \cite{Ros98}. We explain here a short combinatorial proof of the associativity. The following lemma is well-known.

\begin{lemma}\label{Lem:length}
Let $\s,\tau\in\mathfrak{S}_n$ such that $l(\s\tau)=l(\s)+l(\tau)$. Then $T_{\s\tau}=T_\s T_\tau$.
\end{lemma}

For the associativity, like in the shuffle product case, it suffices to show that
$$\sum_{t\in\mathfrak{S}_{j,n-j}}\sum_{s\in\mathfrak{S}_{i,j-i}}T_tT_s=\sum_{t\in\mathfrak{S}_{i,n-i}}\sum_{s\in\mathfrak{S}_{j-i,n-j}^{\uparrow i}}T_tT_s.$$
For any $t\in\mathfrak{S}_{j,n-j}$ and $s\in\mathfrak{S}_{i,j-i}$, since both $t$ and $s$ have only one descent and they are at different positions, $l(ts)=l(t)+l(s)$; for any $\tau\in\mathfrak{S}_{i,n-i}$ and $\s\in\mathfrak{S}_{j-i,n-j}^{\uparrow i}$, the same reason implies $l(\tau\s)=l(\tau)+l(\s)$. By Lemma \ref{Lem:length}, the desired formula is a direct consequence of the formula (\ref{eq1}).
\par

We let $T_\s(V)$ denote the tensor space $T(V)$ endowed with the quantum shuffle product; the sub-algebra of $T_\s(V)$ generated by $V$ is called the quantum shuffle algebra (\cite{Ros95}, \cite{Ros98}); it will be denoted by $S_\s(V)$.

\begin{example}\label{Ex:S(V)}
When $\s:V\ts V\ra V\ts V$ is the flip: for any $v,w\in V$, $\s(v\ts w)=w\ts v$, the quantum shuffle product reduces to the shuffle product and the quantum shuffle algebra $S_\s(V)$  is the symmetric algebra $S(V)$ on $V$.
\end{example}

\subsection{Quantum quasi-shuffle algebras}
The quantum quasi-shuffle product is introduced in \cite{JRZ11} as a particular case of the construction in \cite{JR12}: it can be looked as either a quantization of the quasi-shuffle product \cite{Hoff} or a multiplicative deformation of the quantum shuffle product.
\par
Concretely, the construction of the quantum quasi-shuffle product associates to a braided algebra $(V,\s,m)$ a multiplicative structure (the quantum quasi-shuffle product) on the tensor space $T(V)$ involving the multiplication $m$ via the universal property (\cite{JR12},\cite{JRZ11}). The output is denoted by $T_{\s,m}(V)$ and its sub-algebra $Q_{m,\s}(V)$ generated by $V$ is called the quantum quasi-shuffle algebra. When the multiplication is identically zero, a braided algebra is nothing but a braided vector space and the associated quantum quasi-shufle algebra $Q_{\s,m}(V)$ coincides with the quantum shuffle algebra $S_\s(V)$.
\par
As the (quantum) shuffle product, the quantum quasi-shuffle product admits a combinatorial interpretation via lifting shuffles to a proper braid-type group \cite{F13}. We give a brief recollection on this construction.
\par
We introduced in \cite{F13} the bubble decomposition of elements in $\mathfrak{S}_n$: it is a set-theoretical bijection
$$\mathfrak{S}_n\overset{\sim}{\longrightarrow}\mathfrak{S}_{n-1,1}\times \mathfrak{S}_{n-2,1}\times\cdots\times \mathfrak{S}_{2,1}\times \mathfrak{S}_{1,1}$$
where $\mathfrak{S}_{k,1}=\mathfrak{S}_{k,1}^{\uparrow 0}$. For $\s\in\mathfrak{S}_n$, we denote $(\s^{(n-1)},\s^{(n-2)},\cdots,\s^{(1)})$ its image under the above bijection. If $\s^{(k)}\neq e$, we write
$$\s^{(k)}=s_{t_k(\s)}s_{t_k(\s)+1}\cdots s_k\in\mathfrak{S}_{k,1}$$
where $s_t$ is the transposition $(t,t+1)$ in $\mathfrak{S}_n$ and $1\leq t_k(\s)\leq k$. If $\s^{(k)}=e$, we set $t_k(\s)=0$.
\par
The bubble decomposition can be applied to obtain a reduced expression of a permutation in $\mathfrak{S}_n$: $\s=\s^{(n-1)}\s^{(n-2)}\cdots\s^{(1)}$ is a reduced expression of $\s$.
\par
A generalization of the Matsumoto-Tits section lifting a $(p,q)$-shuffle to $\mc[\GVB_n^+]$ is constructed in \cite{F13} in aim of encoding the quantum quasi-shuffle product. For any $p,q$ as above satisfying $p+q=n$, the generalized Matsumoto-Tits section is a map
$$Q_{p,q}:\mathfrak{S}_{p,q}\ra\mc[{\GVB}_n^+],$$
$$\s\mapsto\prod_{1\leq k\leq n}^{\s^{(k)}\neq e}(\s_{t_k}+(1-\delta_{t_k+1,t_{k+1}})\s_1\cdots\s_{t_k-1}\xi_{t_k})\s_{t_k+1}\cdots \s_k,$$
where $t_k=t_k(\s)$ and the product is executed in a descending way. It should be pointed out that the formula above is slightly different from the one in \cite{F13}: they are related by replacing $\s_1\cdots\s_{i-1}\xi_i$ by $\xi_i$. In fact, in the application to the quantum quasi-shuffle algebras, when acting on the tensor powers of a braided algebra, they give the same result after applying the deleting operator.
\par
These maps $Q_{p,q}$ can be trivially extended to the whole group $\mathfrak{S}_n$ by requiring that it takes zero value on elements not in $\mathfrak{S}_{p,q}$. It turns out to be a linear map $Q_{p,q}:\mc[\mathfrak{S}_n]\ra\mc[\GVB_n^+]$.
\par
These generalized Matsumoto-Tits sections $Q_{p,q}$ can be used to characterize the quantum quasi-shuffle product.
\par
Let $(\widehat{V},\s,m)$ be a braided algebra. We suppose moreover that\begin{enumerate}
\item there exists a specified element $\one\in\widehat{V}$; let $V$ be a linear complement of the subspace $\mc.\one$;
\item for any $v\in\widehat{V}$, 
\begin{equation}\label{eq2}
m(\one\ts v)=m(v\ts\one)=v,\ \ \s(\one\ts v)=v\ts\one\ \ \text{and} \ \s(v\ts\one)=\one\ts v;
\end{equation}
\item the braid $\s$ and the multiplication $m$ can be restricted to $V$ to yield a braided algebra $(V,\s,m)$.
\end{enumerate}  
Let $\s^m:\widehat{V}\ts \widehat{V}\ra \widehat{V}\ts \widehat{V}$ be the linear map defined by: for any $v,w\in \widehat{V}$,
$$\s^m(v\ts w)=\one\ts m(v\ts w).$$
The axioms of $(\widehat{V},\s,m)$ being a braided algebra are equivalent (\cite{F13}, Proposition 4.6) to the existence of a representation of $\mc[\GVB_3^+]$ on $\widehat{V}^{\ts 3}$ defined by:
$$\s_1\mapsto \s\ts\id,\ \ \s_2\mapsto\id\ts\s,\ \ \xi_1\mapsto \s^m\ts\id,\ \ \xi_2\mapsto\id\ts\s^m.$$
If $(\widehat{V},\s,m)$ is a braided algebra, for any $n\geq 2$, $\widehat{V}^{\ts n}$ admits a representation of $\mc[\GVB_n^+]$ via
$$\s_i\mapsto {\id}^{\ts (i-1)}\ts\s\ts{\id}^{\ts(n-i-1)}\ \ \text{and}\ \ \xi_i\mapsto {\id}^{\ts (i-1)}\ts\s^m\ts{\id}^{\ts(n-i-1)}.$$
Let $Q_{\s,m}(V)$ be the quantum quasi-shuffle algebra associated to the braided algebra $(V,\s,m)$ (\cite{JRZ11}) and $\ast$ be the multiplication in $Q_{\s,m}(V)$. A pure tensor $v_1\ts\cdots\ts v_k\in T(\widehat{V})$ is called \emph{deletable} if for any $i=1,\cdots,k$, either $v_i=\one$ or $v_i\in V$. The deleting operator $\bold{D}:T(\widehat{V})\ra T(V)$ is the linear map deleting all $\one$ in a deletable pure tensor. For example, if $v,w\in V$, $\bold{D}(v\ts\one\ts w)=v\ts w$.

\begin{theorem}[\cite{F13}]\label{Thm:1}
For any integers $p,q>0$ and $v_1,\cdots,v_{p+q}\in V$,
$$(v_1\ts\cdots\ts v_p)\ast(v_{p+1}\ts\cdots\ts v_{p+q})=\bold{D}\circ \sum_{s\in\mathfrak{S}_{p,q}}Q_{p,q}(s)(v_1\ts\cdots\ts v_{p+q}),$$
where the tensor product on the right hand side is looked in $\widehat{V}^{\ts(p+q)}$. 
\end{theorem}

\section{Lifting permutations of two descents and combinatorial associativity}

One of the goals of this paper is to lift an arbitrary permutation to the algebra of the generalized virtual braid monoid, generalizing the Matsumoto-Tits section. The main idea is to iterate the lifting of shuffles according to the descents of a permutation; the difficulties lie in the case of two descents: the general case would become transparent thereafter.

\subsection{Lift permutations of two descents}

For $\s\in\mathfrak{S}_n$, the descent set of $\s$ is defined by
$${\des}_n(\s)=\{1\leq i\leq n-1|\ \s^{-1}(i)>\s^{-1}(i+1)\}.$$
For any subset $I\subset\{1,\cdots,n-1\}$, let
$${\des}_n(I)=\{\s\in\mathfrak{S}_n|\ {\des}_n(\s)=I\}$$
denote the set of permutations having descents exactly at positions prescribed in $I$. We will also use the notation
$${\des}_n(\leq I)=\{\s\in\mathfrak{S}_n|\ {\des}_n(\s)\subset I\}=\bigcup_{J\subset I}{\des}_n(J).$$

\begin{example}
\begin{enumerate}
\item For $1\leq i\leq n-1$, it is easy to see from definition that $\{e\}\cup{\des}_n(\{i\})=\mathfrak{S}_{i,n-i}$ is the set of $(i,n-i)$-shuffles.
\item For $n=3$, $I=\{1,2\}$, ${\des}_n(I)=\{s_1s_2s_1\}$ and $\des_n(\leq I)=\mathfrak{S}_3$.
\item For $n=4$, $I=\{1,3\}$, 
$${\des}_4(I)=\{s_1s_3,s_2s_3s_1,s_3s_2s_1s_3,s_1s_2s_3s_1,s_3s_2s_1s_2s_3\},$$
$${\des}_4(\leq I)={\des}_4(I)\cup\{e,s_1,s_2s_1,s_3s_2s_1,s_3,s_2s_3,s_1s_2s_3\}.$$
\item Let $w_0$ be the longest element in $\mathfrak{S}_n$. For $I=\{1,\cdots,n-1\}$, ${\des}_n(I)=\{w_0\}$ and ${\des}_n(\leq I)=\mathfrak{S}_n$.
\end{enumerate}
\end{example}

From now on we fix the following situation: let $I=\{i,j\}\subset \{1,\cdots,n-1\}$ such that $i<j$ and $\s\in{\des}_n(I)$. Since $\s$ has exactly two descents at positions $i$ and $j$ and ${\des}_n(\{i\})=\mathfrak{S}_{i,n-i}\backslash\{e\}$, there exists a map:
$${\des}_n(I)\ra\mathfrak{S}_{j,n-j}\times\mathfrak{S}_{i,j-i}$$
where $\mathfrak{S}_{i,j-i}=\mathfrak{S}_{i,j-i}^{\uparrow 0}$. Here $\mathfrak{S}_{i,j-i}$ encodes the descent at the position $i$ and $\mathfrak{S}_{j,n-j}$ corresponds to the descent at the position $j$. We denote $(\s_L,\s_R)$ the image of $\s\in{\des}_n(I)$ under this map. It is clear that this map is injective but not surjective.

\begin{lemma}\label{Lem:bij2}
There exists a bijection 
$${\des}_n(\leq I)\overset{\sim}{\longrightarrow}\mathfrak{S}_{j,n-j}\times \mathfrak{S}_{i,j-i},$$
extending the above map.
\end{lemma}

The proof of this lemma will be given later (Lemma \ref{Lem:decomp}), as a special case of a general result on permutations of arbitrary descents.
\par

For an element $\tau\in\GVB_n^+$, we can define three length functions:
\begin{enumerate}
\item the braid length $l_s(\tau)$ counting the appearance of the $\s_i$ in any expression of $\tau$ for $i=1,\cdots,n-1$;
\item the virtual length $l_v(\tau)$ counting the appearance of the $\xi_i$ in any expression of $\tau$ for $i=1,\cdots,n-1$;
\item the total length $l(\tau)=l_s(\tau)+l_v(\tau)$.
\end{enumerate}
For $p+q=n$, we define a bilinear map
$${\mm}_{p,q}:\mc[\mathfrak{S}_n]\times\mc[{\GVB}_n^+]\ra\mc[{\GVB}_n^+]$$
as follows: for $\s\in\mathfrak{S}_n$ and $\tau\in\GVB_n^+$,
\[{\mm}_{p,q}(\s,\tau)=\left\{
\begin{matrix}
Q^{\ra l_v(\tau)}_{p-l_v(\tau),q}(\s)\tau\ \ \ & \text{when}\  p-l_v(\tau)>0\ \text{and}\ \s\in\mathfrak{S}_{p-l_v(\tau),q}^{\uparrow l_v(\tau)},\\
0\ \ \  & \text{otherwise},
\end{matrix} \right.
\]
where the notation $Q_{p-l_v(\tau),q}^{\ra l_v(\tau)}$ stands for the action of $Q_{p-l_v(\tau),q}$ on the positions $l_v(\tau)+1,\cdots,n$ (recall that the maps $Q_{p,q}$ have been extended to the entire $\mc[\mathfrak{S}_n]$ by zero.)

Explicitly, if $p-l_v(\tau)>0$ and $\s\in\mathfrak{S}_{p-l_v(\tau),q}^{\uparrow l_v(\tau)}$, 
$${\mm}_{p,q}(\s,\tau)=\left(\prod_{1\leq k\leq n}^{\s^{(k)}\neq e}\left(\s_{t_k}+(1-\delta_{t_k+1,t_{k+1}})\s_1\cdots\s_{t_k-1}\xi_{t_k}\right)\s_{t_{k+1}}\cdots\s_k\right)\tau.$$
We write $\mm_{p,q}(\s):=\mm_{p,q}(\s,e)$ for $\s\in\mathfrak{S}_n$. These maps allow us to define the lifting of permutations of two descents:
$$Q_I:{\des}_n(I)\ra\mc[{\GVB}_n^+],\ \ Q_I(\s)={\mm}_{j,n-j}(\s_L,{\mm}_{i,j-i}(\s_R)).$$
Together with $Q_{p,q}$, we obtain a map $Q_{\leq I}:\des_n(\leq I)\ra \mc[\GVB_n^+]$ by lifting permutations of one descent via $Q_{p,q}$ and two descents via $Q_I$.

\begin{theorem}\label{Thm:2}
Let $v_1,\cdots,v_n\in V$. Then
$$((v_1\ts\cdots\ts v_i)\ast(v_{i+1}\ts\cdots\ts v_j))\ast(v_{j+1}\ts\cdots\ts v_n)=\bold{D}\circ\sum_{\s\in\des_n(\leq I)}Q_{\leq I}(\s)(v_1\ts\cdots\ts v_n),$$
where the tensor product on the right hand side is looked in $\widehat{V}^{\ts n}$.
\end{theorem}

\begin{proof}
We claim that it suffices to show: for $s\in\mathfrak{S}_{i,j-i}$,
$$\bold{D}((\bold{D}\circ Q_{i,j-i}(s)(v_1\ts\cdots\ts v_j))\ast(v_{j+1}\ts\cdots\ts v_n))=\bold{D}\circ\sum_{t\in\mathfrak{S}_{j,n-j}}\bold{M}_{j,n-j}(t,\bold{M}_{i,j-i}(s))(v_1\ts\cdots\ts v_n).$$
Indeed, once this is established, we have
\begin{eqnarray*}
& & ((v_1\ts\cdots\ts v_i)\ast (v_{i+1}\ts\cdots\ts v_j))\ast (v_{j+1}\ts\cdots\ts v_n) \\
&=& \bold{D}\left(\left(\bold{D}\left(\sum_{s\in\mathfrak{S}_{i,j-i}} Q_{i,j-i}(s)(v_1\ts\cdots\ts v_j)\right)\right)\ast(v_{j+1}\ts\cdots\ts v_n)\right)\\
&=& \bold{D}\left(\sum_{s\in\mathfrak{S}_{i,j-i}}\sum_{t\in\mathfrak{S}_{j,n-j}}\bold{M}_{j,n-j}(t,\bold{M}_{i,j-i}(s))(v_1\ts\cdots\ts v_n)\right)\\
&=& \bold{D} \left(\sum_{\s\in{\des}_n(\leq I)}Q_{\leq I}(\s)(v_1\ts\cdots\ts v_n)\right)
\end{eqnarray*}
where the above formula, Lemma \ref{Lem:bij2} and the definition of $Q_{\leq I}$ are applied.
\par
The proof of the above formula is executed in several steps:
\begin{enumerate}
\item Since $s\in\mathfrak{S}_{i,j-i}$, by definition, $\bold{M}_{i,j-i}(s)=Q_{i,j-i}(s)$. Let $\tau\in\GVB_n^+$ be a summand of $Q_{i,j-i}(s)$. It suffices to show that
$$\bold{D}((\bold{D}\circ \tau(v_1\ts\cdots\ts v_j))\ast (v_{j+1}\ts\cdots\ts v_n))=\bold{D}\left(\sum_{t\in\mathfrak{S}_{j,n-j}}\bold{M}_{j,n-j}(t,\tau)(v_1\ts\cdots\ts v_n)\right).$$
\item By definition, the right hand side of the above formula is 
$$\bold{D}\left(\sum_{t\in\mathfrak{S}_{j,n-j}}Q_{j-l_v(\tau),n-j}^{\ra l_v(\tau)}(t)\tau (v_1\ts\cdots\ts v_n)\right).$$
\item We show that $\bold{D}\circ\tau(v_1\ts\cdots\ts v_j)\in T^{j-l_v(\tau)}(V)$. Since $\tau\in\GVB_n^+$ is a summand of $Q_{i,j-i}(s)$, by definition, $\tau(v_1\ts\cdots\ts v_j)$ contains those pure tensors of form $\bold{1}^{\ts l_v(\tau)}\ts w_1\ts\cdots\ts w_{j-l_v(\tau)}$ where $w_1,\cdots,w_{j-l_v(\tau)}\in V$ as summands: indeed, it suffices to notice that each appearance of $\xi_k$ is followed by $\s_1\cdots\s_{k-1}$ which pushes the $\bold{1}$ arising from the action of $\xi_k$ to the first tensorand. Therefore the deleting operator sends $\tau(v_1\ts\cdots\ts v_j)$ to $T^{j-l_v(\tau)}(V)$.
\item By the point (3) and Theorem \ref{Thm:1},
\begin{eqnarray*}
& &(\bold{D}\circ \tau(v_1\ts\cdots\ts v_j))\ast (v_{j+1}\ts\cdots\ts v_n)\\
&=& \sum_{t\in\mathfrak{S}_{j-l_v(\tau),n-j}}Q_{j-l_v(\tau),n-j}(t) ((\bold{D}\circ \tau(v_1\ts\cdots\ts v_j))\ts v_{j+1}\ts\cdots\ts v_n)\\
&=&  \bold{D}\left(\sum_{t\in\mathfrak{S}_{j,n-j}}Q_{j-l_v(\tau),n-j}^{\ra l_v(\tau)}(t) \tau(v_1\ts\cdots\ts v_n)\right).\\
\end{eqnarray*}
The proof is then complete.
\end{enumerate}
\end{proof}

\begin{example}
We take $n=3$, $I=\{1,2\}$ and compute the image of $Q_{\leq I}$. First, $\des_3(\leq I)=\mathfrak{S}_3$, and the decomposition above is $$\mathfrak{S}_3\overset{\sim}{\longrightarrow}\mathfrak{S}_{2,1}\times\mathfrak{S}_{1,1},$$
which is nothing but the bubble decomposition of $\mathfrak{S}_3$:
$$e\mapsto (e,e),\ \ s_1\mapsto (e,s_1),\ \ s_2\mapsto (s_2,e),\ \ s_1s_2\mapsto (s_1s_2,e),$$
$$s_2s_1\mapsto (s_2,s_1),\ \ s_1s_2s_1\mapsto (s_1s_2,s_1).$$
The images of these elements under $Q_I$ are:
$$Q_I(s_1s_2s_1)={\mm}_{2,1}(s_1s_2,{\mm}_{1,1}(s_1))={\mm}_{2,1}(s_1s_2,\s_1+\xi_1)=\s_1\s_2\s_1+\xi_1\s_2\s_1,$$
for the last equality, notice that $\mm_{2,1}(s_1s_2,\xi_1)=0$ since it is $Q_{1,1}(s_1s_2)\xi_1=0$ by definition.
We obtain similarly
$$Q_{\leq I}(s_2s_1)={\mm}_{2,1}(s_2,{\mm}_{1,1}(s_1))=(\s_2+\s_1\xi_2)\s_1+Q_{1,1}^{\ra 1}(s_2)\xi_1=\s_2\s_1+\s_1\xi_2\s_1+\s_1\xi_2\xi_1+\s_2\xi_1.$$
$$Q_{\leq I}(s_1s_2)={\mm}_{2,1}(s_1s_2,{\mm}_{1,1}(e))=\s_1\s_2+\xi_1\s_2,$$
$$Q_{\leq I}(s_1)=\s_1+\xi_1,\ \ Q_{\leq I}(s_2)=\s_2+\s_1\xi_2,\ \ Q_{\leq I}(e)=e.$$
For $v_1,v_2,v_3\in V$, the iterated product $(v_1\ast v_2)\ast v_3$ can be computed inductively by definition:

\begin{eqnarray*}
(v_1\ast v_2)\ast v_3 &=& (v_1\ts v_2)\ast v_3+(\s_1(v_1\ts v_2))\ast v_3+\bold{D}(\one\ts (m(v_1\ts v_2)\ast v_3))\\
&=& \bold{D} ((e+\s_2+\s_1\xi_2+\s_1\s_2+\xi_1\s_2+\s_1+\s_2\s_1+\s_1\xi_2\s_1+\s_1\s_2\s_1+\\
& &\ \ +\xi_1\s_2\s_1+\xi_1+\s_2\xi_1+\s_1\xi_2\xi_1)(v_1\ts v_2\ts v_3)).
\end{eqnarray*}
Comparing elements contained in $\mc[\GVB_3^+]$ gives
$$(v_1\ast v_2)\ast v_3=\bold{D}\left(\sum_{\s\in{\des}_3(\leq\{1,2\})}Q_{\leq I}(\s)(v_1\ts v_2\ts v_3)\right).$$
\end{example}
\begin{example}
We study a more complicated example by taking $n=4$ and $I=\{1,3\}$. First, we enumerate $\des_4(\leq I)$:
\begin{eqnarray*}
{\des}_4(\leq I) &=& \mathfrak{S}_{1,3}\cup \mathfrak{S}_{3,1}\cup{\des}_4(I)\\
&=&\{e,s_1,s_2s_1,s_3s_2s_1,s_3,s_2s_3,s_1s_2s_3,s_1s_3,s_2s_3s_1,s_3s_2s_1s_3,s_1s_2s_3s_1,s_3s_2s_1s_2s_3\}.
\end{eqnarray*}
The computations are listed in the following table:\\
\par

\begin{tabular}{|c|c|c|}
\hline
\emph{Permutation} & \emph{Decomposition} & \emph{Image under $Q_{\leq I}$} \\
\hline
$s_3s_2s_1s_2s_3$ & $(s_1s_2s_3,s_2s_1)$ & $\s_1\s_2\s_3\s_2\s_1+\xi_1\s_2\s_3\s_2\s_1$ \\
\hline
$s_1s_2s_3s_1$ & $(s_1s_2s_3,s_1)$ & $\s_1\s_2\s_3\s_1+\xi_1\s_2\s_3\s_1$ \\
\hline
$s_3s_2s_1s_3$ & $(s_2s_3,s_2s_1)$ & $\s_2\s_3\s_2\s_1+\s_1\xi_2\s_3\s_2\s_1+\s_2\s_3\s_1\xi_2\s_1+\s_1\xi_2\s_3\s_1\xi_1\s_1$ \\
\hline
$s_2s_3s_1$ & $(s_2s_3,s_1)$ & $\s_2\s_3\s_1+\s_1\xi_2\s_3\s_1+\s_2\s_3\xi_1+\s_1\xi_2\s_3\xi_1$ \\
\hline
$s_1s_3$ & $(s_3,s_1)$ & $\s_3\s_1+\s_1\s_2\xi_3\s_1+\s_3\xi_1+\s_1\s_2\xi_3\xi_1$ \\
\hline
$s_1s_2s_3$ & $(s_1s_2s_3,e)$ & $\s_1\s_2\s_3+\xi_1\s_2\s_3$ \\
\hline
$s_2s_3$ & $(s_2s_3,e)$ & $\s_2\s_3+\s_1\xi_2\s_3$ \\
\hline
$s_3$ & $(s_3,e)$ & $\s_3+\s_1\s_2\xi_3$ \\
\hline
$s_3s_2s_1$ & $(s_3,s_2s_1)$ & $\s_3\s_2\s_1+\s_1\s_2\xi_3\s_2\s_1+\s_3\s_1\xi_2\s_1+\s_1\s_2\xi_3\s_1\xi_2\s_1$ \\
\hline
$s_2s_1$ & $(e,s_2s_1)$ & $\s_2\s_1+\s_1\xi_2\s_1$ \\
\hline
$s_1$ & $(e,s_1)$ & $\s_1+\xi_1$ \\
\hline
$e$ & $(e,e)$ & $e$ \\
\hline
\end{tabular}
\\
\par
It is then straightforward to verify that 
$$(v_1\ast(v_2\ts v_3))\ast v_4=\bold{D}\left(\sum_{\s\in{\des}_4(\leq I)}Q_{\leq I}(\s)(v_1\ts v_2\ts v_3\ts v_4)\right).$$
\end{example}

There is another decomposition of ${\des}_n(\leq I)$ with $I=\{i,j\}$ for $1\leq i<j\leq n-1$: it is a bijection
$${\des}_n(\leq I)\overset{\sim}{\longrightarrow}\mathfrak{S}_{i,n-i}\times \mathfrak{S}_{j-i,n-j}^{\uparrow i},\ \ \s\mapsto(\s_L',\s_R').$$
This decomposition gives a different way to lift permutations of two descents. We define a map $Q_I':{\des}_n(I)\ra \mc[\GVB_n^+]$ by:
$$Q_I'(\s)={\mm}_{i,n-i}'(\s_L',{\mm}_{j-i,n-j}'^{\uparrow i}(\s_R'))$$
where
\[{\mm}_{p,q}'(\s,\tau)=\left\{
\begin{matrix}
Q^{\ra l_v(\tau)}_{p,q-l_v(\tau)}(\s)\tau\ \ \ & \text{when}\  q-l_v(\tau)>0\ \text{and}\ \s\in\mathfrak{S}_{p,q-l_v(\tau)}^{\uparrow l_v(\tau)},\\
0\ \ \  & \text{otherwise}.
\end{matrix} \right.
\]
Together with $Q_{p,q}$, we obtain a map $Q_{\leq I}':{\des}_n(\leq I)\ra\mc[\GVB_n^+]$. Similar to the proof of Theorem \ref{Thm:2}, we have the following

\begin{theorem}\label{Thm:2bis}
Let $v_1,\cdots,v_n\in V$. Then
$$(v_1\ts\cdots\ts v_i)\ast((v_{i+1}\ts\cdots\ts v_j)\ast(v_{j+1}\ts\cdots\ts v_n))=\bold{D}\circ\sum_{\s\in\des_n(\leq I)}Q_{\leq I}'(\s)(v_1\ts\cdots\ts v_n),$$
where the tensor product on the right hand side is looked in $\widehat{V}^{\ts n}$.

\end{theorem}

It is easy to see from definition (or Example \ref{Ex:3} below) that $Q_{\leq I}$ and $Q_{\leq I}'$ are different, therefore there is no canonical way to define the generalized Matsumoto-Tits section.

\subsection{Discussions on the associativity of quantum quasi-shuffle products}

It is known that the quantum quasi-shuffle product is associative \cite{JRZ11}, i.e., the left hand sides of the identities in Theorem \ref{Thm:2} and Theorem \ref{Thm:2bis} coincide. 
\par
By combining the associativity of the quantum quasi-shuffle product, Theorem \ref{Thm:2} and Theorem \ref{Thm:2bis}, we obtain:

\begin{corollary}\label{Cor:1}
As elements in $\End(\widehat{V}^{\ts n})$, we have
$$\sum_{\s\in{\des}_n(\leq I)}Q_{\leq I}(\s)=\sum_{\s\in{\des}_n(\leq I)}Q_{\leq I}'(\s).$$
\end{corollary}

Together with Theorem \ref{Thm:1}, we have justified the following observation : the existence of the shuffle product (quantum shuffle product, quantum quasi-shuffle product) relies on the lifting of shuffles (elements with one descent) to a proper braid-type group. When this product is required to be associative, we should study the different liftings of elements with at most two descents (in fact there are two different ways and the associativity is equivalent to the coincidence of their actions on the tensor space).
\par
The identity in Corollary \ref{Cor:1} does not hold in $\mc[\GVB_n^+]$.

\begin{example}\label{Ex:3}
Assume that $n=3$ and $I=\{1,2\}$. The left hand side of the formula in Corollary \ref{Cor:1} is:
$$(e+\s_2+\s_1\xi_2+\s_1\s_2+\xi_1\s_2)(e+\s_1)+\xi_1+\s_2\xi_1+\s_1\xi_2\xi_1,$$
while the right hand side is:
$$(e+\xi_1+\s_1+\s_1\xi_2\s_1+\s_2\s_1)(e+\s_2)+\s_1\xi_2+\s_2\s_1\xi_2+\s_1\xi_2\s_1\xi_2.$$

Notice that they are not the same as elements in $\mc[\GVB_3^+]$, but we will show that their images in $\End(\widehat{V}^{\ts 3})$ coincide. It suffices to verify that in $\End(V^{\ts 3})$ we have 
$$\s_2\xi_1=\s_1\xi_2\s_1\s_2\ \ \text{and}\ \  \s_1\xi_2\xi_1=\s_1\xi_2\s_1\xi_2:$$

\begin{enumerate}
\item for the first identity, $\s_1\xi_2\s_1\s_2=\s_1\s_1\s_2\xi_1=\s_2\xi_1$, where the first equality is the braid-type relation in $\GVB_3$ and the second one holds since for any $v\in \widehat{V}$, $\s^2(\one\ts v)=\one\ts v$ in $\End(\widehat{V}^{\ts 2})$.
\item for the second equality, since $\s^m(\one\ts\one)=\one\ts \one$ and $\s^m(v\ts \one)=\s(v\ts\one)$ for any $v\in \widehat{V}$, we have in $\End(\widehat{V}^{\ts 3})$ the following computations:
$$\s_1\xi_2\xi_1=\s_1\xi_1\xi_2\xi_1=\s_1\xi_2\xi_1\xi_2=\s_1\xi_2\s_1\xi_2.$$
\end{enumerate}
\end{example}

In fact, the two identities above are all obstructions for two sides in Corollary \ref{Cor:1} of being the same in $\mc[\GVB_n^+]$. We explain it in the following remark.

\begin{remark}
Let $W$ be a vector space with a specified element $\bold{1}$. Given a braiding $\s:W\ts W\ra W\ts W$ and an associative operation $m:W\ts W\ra W$ satisfying (\ref{eq2}), we can define a quantum quasi-shuffle operation on $T(W)$ by lifting all shuffles via $Q_{p,q}$ as in Theorem \ref{Thm:1}.
\par
There exist equivalences between:
\begin{enumerate}
\item The quasi-shuffle operation obtained via lifting shuffles is associative;
\item The assignment 
$${\GVB}_3^+\ra\End(W^{\ts 3}),$$
$$\s_1\mapsto \s\ts\id,\ \  \s_2\mapsto \id\ts\s,\ \  \xi_1\mapsto\s^m\ts\id,\ \  \xi_2\mapsto\id\ts\s^m$$ 
defines a representation of $\GVB_3^+$ on $W^{\ts 3}$;
\item $\s$ and $m$ are compatible such that $(W,\s,m)$ is a braided algebra.
\end{enumerate}
This result is a consequence of \cite{JRZ11}, Theorem 4 and \cite{F13}, Proposition 4.6. A direct proof can be executed by establishing a recursive formula using the decomposition 
$$\mathfrak{S}_{p,q}=\mathfrak{S}_{p-1,q}^{\uparrow 1}\cup \mathfrak{S}_{p,q-1}^{\uparrow 1}s_1\cdots s_p.$$
\end{remark}

\section{Lift permutations of arbitrary descents}

We will generalize in this section the construction above to permutations with arbitrary descents. As an application, we will define the total symmetrization operator.

\subsection{Construction}

A $k$-tuple $\underline{c}=(c_1,\cdots,c_k)\in(\mathbb{N}_{>0})^k$ is called a \emph{composition} of $n$ if $c_1+\cdots+c_k=n$; $k$ is called the \emph{length} of the composition and $n$ is its \emph{weight}. Let $C(n)$ denote the set of compositions of $n$.
\par
Let $S(n)$ denote the set of subsets of $\{1,\cdots,n\}$. There exists a bijection
$$\vp:C(n)\ra S(n-1),\ \ (c_1,\cdots,c_k)\mapsto (c_1,c_1+c_2,\cdots,c_1+\cdots+c_{k-1});$$
its inverse is given by: for $\{i_1,\cdots,i_k\}\in S(n-1)$ with $i_1<\cdots<i_k$, its pre-image is 
$$(i_1,i_2-i_1,i_3-i_2,\cdots,i_k-i_{k-1},n-i_k)\in C(n).$$
\par
We fix a subset $I\subset\{1,\cdots,n-1\}$ with $|I|=k$, which will be the descent set of the permutations. We denote $\underline{c}=(c_1,\cdots,c_{k+1})\in C(n)$ such that $\vp(\underline{c})=I$.

\begin{lemma}\label{Lem:decomp}
We have the following bijection whose inverse is given by the multiplication:
$${\des}_n(\leq I)\overset{\sim}{\longrightarrow}\mathfrak{S}_{c_1+\cdots+c_k,c_{k+1}}\times\cdots\times\mathfrak{S}_{c_1+c_2,c_3}\times \mathfrak{S}_{c_1,c_2}.$$
\end{lemma}

\begin{proof}
Since elements in $\mathfrak{S}_{c_1+\cdots+c_l,c_{l+1}}$ for any $l=1,\cdots,k$ have only one descent at the position $c_1+\cdots+c_l$, permutations on the right hand side have descents contained in $I$, therefore the map is well-defined.
\par
The injectivity is clear by definition. To show the bijection, we compare the cardinality: the right hand side has cardinality 
$$\frac{n!}{c_1!\cdots c_{k+1}!}.$$
The counting for the left hand side is standard: see for example \cite{Miklos}, Lemma 1.3 and take account of the definition of the inverse of $\vp$. Therefore the map is bijective.
\end{proof}

For $\s\in{\des}_n(I)$, let $(\s^k,\s^{k-1},\cdots,\s^2,\s^1)$ denote its image in $\mathfrak{S}_{c_1+\cdots+c_k,c_{k+1}}\times\cdots\times\mathfrak{S}_{c_1+c_2,c_3}\times \mathfrak{S}_{c_1,c_2}$. We define a map $Q_I:\des_n(\leq I)\ra\mc[\GVB_n^+]$ by:
$$Q_I(\s)={\mm}_{c_1+\cdots+c_k,c_{k+1}}(\s^k,{\mm}_{c_1+\cdots+c_{k-1},c_k}(\s^{k-1},\cdots,{\mm}_{c_1,c_2}(\s^1))\cdots)$$
which is the iterated action of $\mm_{p,q}$ according to the decomposition in Lemma \ref{Lem:decomp}, as in the case of two descents.
\par
We write in the following theorem $v_{i_1}\cdots v_{i_t}:=(v_{i_1}\ts\cdots\ts v_{i_t})$ for $v_{i_1},\cdots,v_{i_t}\in V$.

\begin{theorem}\label{Thm:main2}
For any $v_1,\cdots,v_n\in V$, we have:
$$((\cdots((v_1\cdots v_{c_1}\ast v_{c_1+1}\cdots v_{c_1+c_2})\ast v_{c_1+c_2+1}\cdots v_{c_1+c_2+c_3})\ast\cdots)\ast v_{c_1+\cdots+c_k+1}\cdots v_n)=$$
$$=\bold{D}\circ\sum_{\s\in{\des}_n(\leq I)}Q_I(\s)(v_1\ts\cdots \ts v_n),$$
where the tensor product on the right hand side is looked in $\widehat{V}^{\ts n}$.
\end{theorem}

\begin{proof}
Apply Theorem \ref{Thm:2} repeatedly.
\end{proof}

By introducing the notation 
$$D_{\leq I}=\sum_{\s\in{\des}_n(\leq I)}\s,$$
the right hand side can be written as
$$\bold{D}\circ Q_I(D_{\leq I})(v_1\ts\cdots\ts v_n).$$
The decomposition of $\des_n(\leq I)$ into $k$ shuffles in Lemma \ref{Lem:decomp} is not unique: for example, there exists a bijection
$${\des}_n(\leq I)\overset{\sim}{\longrightarrow}\mathfrak{S}_{c_1,c_2+\cdots+c_{k+1}}\times\cdots\times\mathfrak{S}_{c_{k-1},c_k+c_{k+1}}^{\uparrow c_1+\cdots+c_{k-2}}\times \mathfrak{S}_{c_k,c_{k+1}}^{\uparrow c_1+\cdots+c_{k-1}}.$$
The associativity of the quantum quasi-shuffle product guarantees that the image of the lifted element in $\End(\widehat{V}^{\ts n})$ is independent of the choice of the decomposition of ${\des}_n(\leq I)$ into $k$ shuffles. 
\par
For any decomposition of ${\des}_n(\leq I)$ into $k$ shuffles, there exists a lifting of elements in ${\des}_n(I)$ to $\mc[\GVB_n^+]$ as above and an analog of Theorem \ref{Thm:main2}.

\begin{corollary}\label{Cor:same}
Let $Q_I$ and $Q_I'$ be two different liftings of ${\des}_n(\leq I)$ according to the two different decompositions of ${\des}_n(\leq I)$ into $k$ shuffles. Then the images of $Q_I(D_{\leq I})$ and $Q_I'(D_{\leq I})$ in $\End(\widehat{V}^{\ts n})$ coincide.
\end{corollary}

\subsection{Application to the total symmetrization operator}

We consider the special case $I=\{1,\cdots,n-1\}$ in the above discussion. By definition, $\vp(\underline{c})=I$ where $\underline{c}=(1,1,\cdots,1)\in C(n)$. It is clear that $\des_n(\leq I)=\mathfrak{S}_n$. 

\begin{definition}
The map $Q_I:\mathfrak{S}_n\ra\mc[\GVB_n^+]$ is called a generalized Matsumoto-Tits section, we will denote its linear extension by $Q:\mc[\mathfrak{S}_n]\ra\mc[\GVB_n^+]$.
\end{definition}

As in the case of two descents, there is no canonical way to define the generalized Matsumoto-Tits section.
\par
By Theorem \ref{Thm:main2} and Corollary \ref{Cor:same}, for any $v_1,\cdots,v_n\in V$,
$$v_1\ast\cdots\ast v_n=\bold{D}\circ Q\left(\sum_{\s\in\mathfrak{S}_n}\s\right)(v_1\ts\cdots\ts v_n).$$

\begin{definition}
\begin{enumerate}
\item The element $Q_n:=\sum_{\s\in\mathfrak{S}_n}Q(\s)$ is called a total symmetrization element in $\mc[\GVB_n^+]$.
\item The element $\QS_n:=\bold{D}\circ Q_n\in\Hom(V^{\ts n},T^{\leq n}(V))$ is called the total symmetrization operator of degree $n$ and the sum $\QS=\sum_{n\geq 0}\QS_n\in\End(T(V))$ is called the total symmetrization operator.
\end{enumerate}
\end{definition}

By definition (\cite{JRZ11}, \cite{JR12}), the quantum quasi-shuffle algebra $Q_{\s,m}(V)$ is the sub-algebra of $T_{\s,m}(V)$ generated by $V$.

\begin{corollary}
We have $Q_{\s,m}(V)=\im \QS$.
\end{corollary}

Another direct definition of the total symmetrization operator is given by Jian \cite{J13} by generalizing the construction of Guo and Keigher (\cite{GK00}), compared to their results, our approach makes the combinatorial aspects of the quantum quasi-shuffle product (quasi-shuffle product, quantum shuffle product) more transparent and can be applied not only to the extremal cases (i.e., $|I|=1$ and $|I|=n-1$) as in \cite{J13} but also to an arbitrary descent. These generalized Matsumoto-Tits sections deserve further studies.

The total symmetrization operator allows us to define the defining ideal in an explicit way.

\begin{definition}
The ideal $I_{\s,m}(V):=\ker \QS\subset T(V)$ is called the defining ideal of the quantum quasi-shuffle algebra $Q_{\s,m}(V)$.
\end{definition}

\begin{remark}
As shown in \cite{JRZ11}, \cite{JR12}, by requiring elements in $V$ to be primitive, by the universal property, there exists a coproduct on $T_{\s,m}(V)$ which can be restricted to $Q_{\s,m}(V)$, making them into braided Hopf algebras. It is clear that with this coproduct, $I_{\s,m}(V)$ is a bi-ideal.
\end{remark}

It is difficult to decide a reasonable generating set of $I_{\s,m}(V)$, even when the multiplication is identically zero (quantum shuffle algebras $S_\s(V)$): in this case it is a question asked by Andruskiewitsch \cite{A03}. Several partial results are known when the braiding is diagonal and $m$ is zero:
\begin{enumerate}
\item The algebra $S_\s(V)$ is finite dimensional (\cite{Ang});
\item The algebra $S_\s(V)$ is infinite dimensional but has finite Gelfand-Kirillov dimension (\cite{Ros98});
\item The algebra $S_\s(V)$ is of infinite Gelfand-Kirillov dimension (\cite{F12}).
\end{enumerate}
We will return to the study of the ideal $I_{\s,m}(V)$ in the next section by showing that a quantum quasi-shuffle algebra $Q_{\s,m}(V)$ can be degenerated to a quantum shuffle algebra $S_\s(V)$ and how relations in $S_\s(V)$ can be lifted to $Q_{\s,m}(V)$.

\section{Degeneration of associative structures and defining ideals}

We recall the following diagram involving various braid groups (\cite{F13}):
\[
\xymatrix{\GVB_n \ar[r]^-\alpha \ar[d]^-\gamma & \mathfrak{B}_n\ar[d]^-\beta \\ \VB_n  & \mathfrak{S}_n,}
\]
where $\alpha$ (resp. $\beta$; $\gamma$) is the quotient by the normal sub-group generated by $\xi_i$ (resp. $\s_i^2$; $\s_i^2$) for $1\leq i\leq n-1$. They induce a similar diagram of monoids associated to these groups. When passed to the algebra level, the diagram above can be completed to a commutative square:
\[
\xymatrix{\mc[{\GVB}_n^+] \ar[r]^-{\wt{\alpha}} \ar[d]^-{\wt{\gamma}} & \mc[\mathfrak{B}_n^+] \ar[d]^-{\wt{\beta}}\\ 
\mc[{\VB}_n^+] \ar[r]^-{\wt{\delta}} & \mc[\mathfrak{S}_n]}.
\]
where $\wt{\alpha}$ (resp. $\wt{\beta}$; $\wt{\gamma}$; $\wt{\delta}$) is the quotient by the ideal generated by $\xi_i$ (resp. $\s_i^2-1$; $\s_i^2-1$; $\xi_i$) for $1\leq i\leq n-1$.

\begin{lemma}\label{Lem:red}
For any $s\in\mathfrak{S}_{p,q}$, $\widetilde{\alpha}(Q_{p,q}(s))=T_s$ and $\widetilde{\beta}(T_s)=s$.
\end{lemma}

\begin{proof}
By the definition of $Q_{p,q}(s)$, 
$$\widetilde{\alpha}(Q_{p,q}(s))=\prod_{1\leq k\leq n}^{\s^{(k)}\neq e}\s_{t_k}\s_{t_k+1}\cdots\s_k.$$
Since the bubble decomposition produces a reduced expression, $\widetilde{\alpha}(Q_{p,q}(s))$ is the image of $s$ under the Matsumoto-Tits section by definition.
\end{proof}

\subsection{Degeneration}
Let $(V,\s,m)$ be a braided algebra. By forgetting structures, $(V,\s)$ is a braided vector space. We let $S_\s(V)$ (resp. $Q_{\s,m}(V)$) denote the quantum shuffle algebra (resp. quantum quasi-shuffle algebra) associated to the braided vector space $(V,\s)$ (resp. the braided algebra $(V,\s,m)$).
\par
The tensor space $T(V)$ has a natural grading by $T^n(V)$, making $T_\s(V)$ a braided algebra and $T_{\s,m}(V)$ a filtered algebra. As sub-algebras, $S_\s(V)$ is a graded algebra and $Q_{\s,m}(V)$ is a filtered algebra with this grading structure. Let $T^{gr}_{\s,m}(V)$ and $Q_{\s,m}^{gr}(V)$ denote the graded algebras associated to the filtered algebras, the canonical projections are denoted by $\wt{\pi}:T_{\s,m}(V)\ra T_{\s,m}^{gr}(V)$ and $\pi:Q_{\s,m}(V)\ra Q_{\s,m}^{gr}(V)$.

\begin{proposition}
There exists an isomorphism of algebra $T_{\s,m}^{gr}(V)\cong T_\s(V)$.
\end{proposition}

\begin{proof}
For $v_1,\cdots,v_{p+q}\in V$, by Theorem \ref{Thm:1}, the product of $v_1\ts\cdots\ts v_p$ and $v_{p+1}\ts\cdots\ts v_{p+q}$ in $T_{\s,m}(V)$ is given by:
$$(v_1\ts\cdots\ts v_p)\ast(v_{p+1}\ts\cdots\ts v_{p+q})=\bold{D}\circ \sum_{s\in\mathfrak{S}_{p,q}}Q_{p,q}(s)(v_1\ts\cdots\ts v_{p+q}).$$

In the above formula, passing to $T_{\s,m}^{gr}(V)$ requires us to drop the pure tensors of degree less than $p+q$; but to obtain these pure tensors it is obliged to act $\xi_i$ for $1\leq i\leq p+q-1$ at least once.
\par
We denote $\overline{v_1\ts\cdots\ts v_n}:=\wt{\pi}(v_1\ts\cdots\ts v_n)$ and the product therein by $\ast^{gr}$. The above argument shows that in $T_{\s,m}^{gr}(V)$,
\begin{eqnarray*}
(\overline{v_1\ts\cdots\ts v_p})\ast^{gr}(\overline{v_{p+1}\ts\cdots\ts v_{p+q}}) &=& \bold{D}\circ\wt{\alpha}\left(Q_{p,q}\left(\sum_{s\in\mathfrak{S}_{p,q}} s\right)\right)(v_1\ts\cdots\ts v_{p+q})\\
&=& \sum_{s\in\mathfrak{S}_{p,q}} T_s(v_1\ts\cdots\ts v_{p+q})\\
&=& (\overline{v_1\ts\cdots\ts v_p})\bullet (\overline{v_{p+1}\ts\cdots\ts v_{p+q}}),
\end{eqnarray*}
where $\bullet$ is the quantum shuffle product and Lemma \ref{Lem:red} is applied. 
\end{proof}

Passing to the sub-algebra generated by degree $1$ we have

\begin{corollary}
The isomorphism in the above proposition induces an isomorphism of algebras $Q_{\s,m}^{gr}(V)\cong S_\s(V)$.
\end{corollary}

We obtain therefore two linear isomorphisms $\wt{\pi}:T_{\s,m}(V)\cong T_\s(V)$ and $\pi:Q_{\s,m}(V)\cong S_\s(V)$.

\subsection{Application: study of defining ideals}
We apply the degeneration procedure to study the lifting of defining relations in $S_\s(V)$ to $Q_{\s,m}(V)$.
\par
For $n\geq 2$, let $T_n=\sum_{s\in\mathfrak{S}_n}T_s$ be the total symmetrization operator in $\mc[\mathfrak{B}_n^+]$, looked in $\mc[\GVB_n^+]$. 
\par
Fix a total symmetrization element $\text{QS}_n\in\mc[\GVB_n^+]$, it can be written as $\QS_n=T_n+W_n$, where $W_n\in\mc[\GVB_n^+]$ satisfies $\wt{\alpha}(W_n)=0$. Denote $I_{\s,m}(V):=\ker \QS\subset T(V)$ be the defining ideal of $Q_{\s,m}(V)$ and $I_{\s,m}(V)_k=I_{\s,m}(V)\cap T^{\leq k}(V)$. Let $I_\s(V)_k:=\ker (T_n:T^n(V)\ra T^n(V))$ and $I_\s(V)=\bigoplus_{k\geq 2} I_\s(V)_k$ be the defining ideal of $S_\s(V)$. 
\par
The proof of the following lemma is clear.

\begin{lemma}\label{Lem:-1}
Let $x\in T^{\leq n}(V)$ satisfy $\QS(x)=0$. We write $x=x_1+x_2$ where $x_1\in T^n(V)$ and $x_2\in T^{\leq n-1}(V)$, then $T_nx_1=0$.
\end{lemma}

Assume that a family of homogeneous generators of $I_\s(V)$ is known. We start with the degree $2$ case: take $x\in I_{\s,m}(V)_2$, i.e., $\QS(x)=0$; by Lemma \ref{Lem:-1}, we may write $x=\overline{x}+v$ for $\overline{x}\in I_\s(V)$ and $v\in V$, the condition $x\in I_{\s,m}(V)_2$ implies that 
$$v=-{\QS}_2(\overline{x})=-T_2(\overline{x})-W_2(\overline{x})=-W_2(\overline{x}),$$
hence $x=\overline{x}-W_2(\overline{x}).$ 
Therefore relations in $I_{\s,m}(V)_2$ can be lifted from those in $I(V)_2$ in an explicit way.
\par
Suppose that relations in $I_{\s,m}(V)_k$ for $k=2,\cdots,n-1$ can be lifted from those in $I_\s(V)_2,\cdots,I_\s(V)_k$ in an explicit way. 
\par
We take $x\in I_{\s,m}(V)_n$, by Lemma \ref{Lem:-1}, $\QS(x)=0$ allows us to write $x=\overline{x}+x'$ where $\overline{x}\in I_\s(V)_n$ and $x'\in T^{\leq n-1}(V)$, and then 
$$\QS(\overline{x})=-{\QS}(x').$$ 
Since $T_n(\overline{x})=0$ and $\QS_n=T_n+W_n$, $W_n(\overline{x})=-\QS(x')$ and $x'$ can be determined up to $I_{\s,m}(V)_{n-1}$. To be precise, take any $y\in \QS^{-1}(W_n(-\overline{x}))$, then 
$$x-\overline{x}-y\in I_{\s,m}(V)_{n-1},\ \ (\ \text{i.e., }{\QS}(x-\overline{x}-y)=0\ )$$ 
and the inductive algorithm for rank $n-1$ can be applied. As a summary, by an inductive procedure, elements in $I_{\s,m}(V)_{n}$ can be lifted from those in $I_{\s}(V)_2,\cdots,I_{\s}(V)_n$ using the above algorithm.
\par
We terminate this argument by the following example:

\begin{example}
Let $V=\sspan\{y_i|\ i\in\mathbb{N}_{>0}\}$ be an associative algebra where the multiplication $m:V\ts V\ra V$ is determined by
$$m(y_i\ts y_j)=y_{i+j}$$
for any $i,j\geq 1$. Then $(V,P,m)$ is a braided algebra without unit where for any $u,v\in V$,
$$P:V\ts V\ra V\ts V,\ \ u\ts v\mapsto v\ts u$$ 
is the flip.
\par
The quantum quasi-shuffle algebra associated to this braided algebra is the quasi-shuffle algebra \cite{Hoff}; we let $C_m(V)$ denote it. This algebra is isomorphic to the algebra $\mathfrak{H}^1$ in the theory of multiple zeta values (MZVs) (\textit{loc. cit}).
\par
By the above argument, there exists a linear isomorphism $C_m(V)\cong S(V)$ (see Example \ref{Ex:S(V)}). We let $I_m(V)$ and $I(V)$ denote the defining ideals of them. The ideal $I(V)$ is clearly generated by $y_i\ts y_j-y_j\ts y_i$ for any $i\neq j\geq 1$, hence the ideal $I_m(V)$ is generated by
$$y_i\ts y_j-y_j\ts y_i-{\QS}_2(y_i\ts y_j-y_j\ts y_i)=y_i\ts y_j-y_j\ts y_i.$$
Therefore $C_m(V)$ is a commutative algebra, isomorphic to the infinite polynomial algebra $\mc[y_1,y_2,\cdots]$, as shown in \cite{Hoff}.
\end{example}

The argument in this example can be applied to prove the following result, generalizing Theorem 15 in \cite{JRZ11}.

\begin{proposition}
Let $(V,\s,m)$ be a braided algebra which is moreover braided commutative (i.e., $m\circ\s=m$). Assume that $S_\s(V)$ is commutative, then $Q_{\s,m}(V)$ is commutative.
\end{proposition}

\begin{proof}
Take a basis $\{y_i\}$ of $V$ and consider the degeneration $Q_{\s,m}(V)\ra S_\s(V)$. As $S_\s(V)$ is commutative, the lifting of the commutation relation between $y_i$ and $y_j$ is given by 
$$y_i\ts y_j-y_j\ts y_i-{\QS}_2(y_i\ts y_j-y_j\ts y_i)=y_i\ts y_j-y_j\ts y_i$$
in $Q_{\s,m}(V)$, thanks to the braided commutativity.
\end{proof}

In particular, when $\s^2=\id$, $S_\s(V)$ is commutative, this is the case of Theorem 15 in \cite{JRZ11}.
\par
Starting from a quantum quasi-shuffle algebra, we obtain a quantum shuffle algebra by degeneration; this procedure can be also reversed to study the classification of multiplicative deformations of a braided vector space. We plan to return to this topic in a further publication.

\section*{Acknowledgements}
The author is supported by the Alexander von Humboldt Foundation.

\end{document}